\documentclass{amsart}
\usepackage{amssymb,latexsym,amsthm,amsmath,hyperref}
\numberwithin{equation}{section}
\allowdisplaybreaks

\def\gc{\mathfrak{c}}

\def\ge{\mathfrak{e}}
\def\gf{\mathfrak{f}}
\def\gg{\mathfrak{g}}
\def\gh{\mathfrak{h}}

\def\gk{\mathfrak{k}}

\def\gp{\mathfrak{p}}

\def\gs{\mathfrak{s}}
\def\gt{\mathfrak{t}}
\def\gu{\mathfrak{u}}

\def\B{\mathbb{B}}
\def\C{\mathbb{C}}
\def\D{\mathbb{D}}

\def\N{\mathbb{N}}

\def\R{\mathbb{R}}
\def\T{\mathbb{T}}

\def\cA{\mathcal{A}}

\def\cF{\mathcal{F}}

\def\cH{\mathcal{H}}

\def\cL{\mathcal{L}}

\def\cT{\mathcal{T}}

\newcommand{\Ad}{\mathrm{Ad}}
\newcommand{\ad}{\mathrm{ad}}
\newcommand{\End}{\mathrm{End}}

\newcommand{\SO}{\mathrm{SO}}

\newcommand{\SU}{\mathrm{SU}}
\newcommand{\U}{\mathrm{U}}
\newcommand{\Tr}{\mathrm{Tr}\, }
\newcommand{\so}{\mathfrak{so}}
\newcommand{\su}{\mathfrak{su}}
\newcommand{\gsp}{\mathfrak{sp}}
\newcommand{\gsl}{\mathfrak{sl}}

\newcommand{\GL}{\mathrm{GL}}
\newcommand{\bs}{\mathbf{s}}
\newcommand{\bt}{\mathbf{t}}

\theoremstyle{plain}

\newtheorem{theorem}{Theorem}

\newtheorem{lemma}[theorem]{Lemma}

\newtheorem{corollary}[theorem]{Corollary}

\newtheorem{proposition}[theorem]{Proposition}

\theoremstyle{definition}

\theoremstyle{remark}
\newtheorem{example}[theorem]{Example}

\newtheorem{remark}[theorem]{Remark}

\numberwithin{theorem}{section}

\begin{document}

\title[Commuting Toeplitz Operators on Bounded Symmetric Domains]{Commuting Toeplitz Operators on Bounded Symmetric Domains and Multiplicity-Free Restrictions of Holomorphic Discrete Series}

\author{Matthew Dawson}
\email{mdawso5@math.lsu.edu}
\address{Department of Mathematics, Louisiana State University, Baton Rouge, LA
70803, U.S.A.}
\thanks{M. Dawson would like to thank the VIGRE program at LSU, DMS-0739382, for support. He would also like to thank CIMAT for its hospitality and support during a visit in September 2013.}

\author{Gestur \'{O}lafsson}
\email{olafsson@math.lsu.edu}
\address{Department of Mathematics, Louisiana State University, Baton Rouge,
LA 70803, U.S.A.}
\thanks{The research of G. \'Olafsson was supported by the NSF grant DMS-1101337.}

\author{Ra\'{u}l Quiroga-Barranco}
\email{quiroga@cimat.mx}
\address{Centro de Investigaci\'{o}n en Matem\'{a}ticas, Guanajuato, Guanajuato, Mexico.}
\thanks{R. Quiroga-Barranco was supported by SNI and the Conacyt Grant 166891.}

\begin{abstract}
    For any given bounded symmetric domain, we prove the existence of commutative $C^*$-algebras generated by Toeplitz operators acting on any weighted Bergman space. The symbols of the Toeplitz operators that generate such algebras are defined by essentially bounded functions invariant under suitable subgroups of the group of biholomorphisms of the domain. These subgroups include the maximal compact groups of biholomorphisms. We prove the commutativity of the Toeplitz operators by considering the Bergman spaces as the underlying space of the holomorphic discrete series and then applying known multiplicity-free results for restrictions to certain subgroups of the holomorphic discrete series. In the compact case we completely characterize the subgroups that define invariant symbols that yield commuting Toeplitz operators in terms of the multiplicity-free property.
\end{abstract}


\maketitle

\section{Introduction}
The weighted Bergman spaces on bounded symmetric domains are a fundamental object in Analysis. They come equipped with a natural projection, the Bergman projection, determined by a reproducing kernel property. This structure allows to consider the so-called Toeplitz operators defined as a multiplier operator followed by the Bergman projection. The special role of Toeplitz operators is observed, for example, in their density in the space of all bounded operators in the strong operator topology (see~\cite{Englis}). It is thus a surprising fact that there are large commutative $C^*$-algebras generated by Toeplitz operators. The existence of such commutative algebras is a current topic of interest in Complex Analysis, as the references of this work show.

The commutative $C^*$-algebras generated by Toeplitz operators known to this date always have an associated distinguished geometry. The latter is given by a subgroup $H$ of the group of biholomorphic maps of the domain. More precisely, for certain choices of $H$ the $H$-invariant essentially bounded functions determine commuting Toeplitz operators. Up to this date this kind of phenomenon has been observed only for the unit ball $\B^n$ with $H$ a maximal Abelian subgroup of its biholomorphisms (see \cite{GQV,QVBall1,QVBall2}) and some variations (see \cite{BVQuasiRadial,QSProjective,QVReinhardt}) as well as for the natural translations of a tube type domain in the weightless case (see \cite{VasilevskiTube}). It is an open problem to find higher rank irreducible bounded symmetric domains that admit for any weight large families of commuting Toeplitz operators acting on their Bergman spaces.

On the other hand, the weighted Bergman spaces are also very important in Harmonic Analysis. In this setup, the Bergman spaces provide the underlying spaces of the holomorphic discrete series representations for noncompact simple Lie groups associated to bounded symmetric domains. The study of such representations is fundamental as part of the picture to understand the unitary representations of simple Lie groups. In particular, a very great deal of attention has been given to the holomorphic discrete series; see for example \cite{FarautKoranyi,HarishChandraIV,Wallach1} just to mention a few.

The representations in the holomorphic discrete series are all irreducible for the action of the whole group of biholomorphisms of the corresponding bounded symmetric domain. However, for a (proper closed) subgroup $H$ of all  biholomorphisms the irreducibility is lost to be replaced by a direct integral decomposition into classes of irreducible unitary representations of $H$. The study of the branching behavior for these representations and many other is a fundamental part of Harmonic Analysis. In particular, for the holomorphic discrete series representations there are very general results that provide some of the subgroups $H$ for which the restricted representation is multiplicity-free, i.e.~so that the classes of irreducible representations over $H$ appear with multiplicity $1$ in the direct integral decomposition. Some of the results of this sort most relevant to this work can be found in \cite{ADO,Kobayashi,OlafssonOrsted}.

The main goal of this work is to prove that subgroups defining multiplicity-free restrictions of the holomorphic discrete series yield commutative $C^*$-algebras generated by Toeplitz operators. To be more precise, we obtain commuting Toeplitz operators when their symbols are invariant under a subgroup with a multiplicity-free restriction of the holomorphic discrete series representation. That is the content of Theorem~\ref{thm:commToeplitz-multiplicityFree}, which holds for arbitrary bounded symmetric domains and for all weighted Bergman spaces.

With the known multiplicity-free results at our disposal we obtain many examples of such commutative $C^*$-algebras. One of our most fundamental results is Theorem~\ref{thm:commToeplitz-maxCompact} which establishes the commutativity of the $C^*$-algebra generated by Toeplitz operators whose symbols are invariant under a maximal compact subgroup of the biholomorphisms of the domain; this is a generalization to arbitrary bounded symmetric domains of the commutativity of the Toeplitz operators with radial symbols proved in \cite{GKVRadial}. We also prove in Theorem~\ref{thm:commToeplitz-symmPairs} that the $C^*$-algebra generated by Toeplitz operators is commutative whenever the symbols for the Toeplitz operators are precisely those left invariant by a symmetric subgroup; this includes all the subgroups whose Lie subalgebras are listed in the second and third column of Table~I. In Theorem~\ref{thm:commToeplitz-tube} we prove the commutativity of the Toeplitz operators whose symbols are invariant under the natural translations in the unbounded realization of a bounded symmetric domain of tube type.

Not only does this solve the problem of the existence of commutative $C^*$-algebras generated by Toeplitz operators for every irreducible bounded symmetric domain (exceptional domains included) on every weighted Bergman space, but it does so with an abundance of examples.

An important problem in the study of the commutative $C^*$-algebras generated by Toeplitz operators is to find an explicit way to simultaneously diagonalize the Toeplitz operators. With this respect, our Theorems~\ref{thm:commToeplitz-antihol} and \ref{thm:commToeplitz-tube} provide an explicit unitary isomorphism that allows to achieve this for many sets of symbols. This includes all the symbols that are invariant under subgroups whose Lie algebras are listed in the third column of Table~I.

Once this many commutative $C^*$-algebras generated by Toeplitz operators have been obtained, it is natural to ask about the classification problem for them. In this direction, our Theorem~\ref{thm:commToeplitz-compact} shows that for symbols invariant under a compact subgroup the problem is equivalent to that of finding the multiplicity-free restrictions: a compact subgroup yields invariant symbols with commuting Toeplitz operators if and only if the restriction to such compact subgroup of the holomorphic discrete series is multiplicity-free.

In general, the problem of finding all the multiplicity-free restrictions is very hard. However, there are some cases where this can be easily verified. Hence, in Example~\ref{ex:torusSU(n,m)} we prove that for the compact subgroup of diagonal matrices (which is maximal Abelian) in $\SU(2,2)$  the restriction of the holomorphic discrete series representation is not multiplicity-free, and so the corresponding $C^*$-algebra generated by Toeplitz operators is not commutative. As proved in \cite{QVBall1,QVBall2} to every maximal Abelian subgroup of $\SU(n,1)$ there corresponds invariant symbols that define commuting Toeplitz operators. Example~\ref{ex:torusSU(n,m)} thus shows that the latter behavior does not extend to higher rank domains.

As for the distinguished geometry observed for commuting Toeplitz operators established in previous works, we can say that such phenomenon is still present since all our subgroups preserve the (Riemannian symmetric) geometry of the bounded symmetric domains. On the other hand, all the examples given in \cite{QVBall1,QVBall2,QSProjective} of commuting Toeplitz operators for the case of the unit ball $\B^n$ and the complex projective space $\mathbb{P}^n(\C)$ are obtained from $n$-dimensional maximal Abelian subgroups whose orbits are (a.e.) $n$-dimensional, flat, Lagrangian and totally geodesic. However, in Theorem~\ref{thm:commToeplitzBn-nonflat} we show that the subgroup $\SO_0(n,1)$ defines invariant symbols that yield commuting Toeplitz operators so that one of the $\SO_0(n,1)$-orbits is $n$-dimensional, non-flat, Lagrangian and totally geodesic. This proves that the maximal Abelian subgroups are not enough to obtain all the $C^*$-algebras generated by Toeplitz operators coming from invariant symbols in the unit ball $\B^n$ even if we require the existence of a Lagrangian and totally geodesic $n$-dimensional orbit. With these new examples, the geometry of the orbits does not even have to be flat.

As for the organization of the paper, in Section~\ref{sec:BergmanSpaces} we introduce the Bergman spaces and their relationship with the holomorphic discrete series. In Section~\ref{sec:interToeplitz} we define the notions of Toeplitz and intertwining operators and observe their properties as they relate to symbols. Section~\ref{sec:multFree} states the basic facts and results on multiplicity-free restrictions that we need to prove Theorem~\ref{thm:commToeplitz-multiplicityFree}. Finally, Sections~\ref{sec:commToeplitz-symmPairs} and \ref{sec:commToeplitz-compact} develop the main results corresponding the case of symmetric pairs and compact subgroups. Also, several examples are given in all sections to clarify our results.

\section{Bergman spaces and the holomorphic discrete series}
\label{sec:BergmanSpaces}
Let $G$ be a connected noncompact simple Lie group with finite center whose Lie algebra has a Cartan decomposition $\gg = \gk + \gp$. We will assume that the center $\gc$ of $\gk$ is nontrivial and so $1$-dimensional. Let us denote by $K$ the maximal compact subgroup of $G$ whose Lie algebra is $\gk$. In particular, the symmetric space $G/K$ is Hermitian. We denote by $G^\C$ the simply connected Lie group with Lie algebra $\gg^\C$. Without loss of generality we can assume that $G$ is a subgroup of $G^\C$.

We proceed to recall the definition of the Bergman spaces on the symmetric space $G/K$. Then, we will relate Bergman spaces to the construction of the holomorphic discrete series of universal covering group $\widetilde{G}$ discussed in \cite{FarautKoranyi, HarishChandraIV, Wallach1}. For the details of our representation theoretic constructions we refer to \cite{FarautKoranyi, HelgasonBook, Satake, Wallach1}, and provide further explanations only when necessary for our arguments. We will mostly follow the Lie theoretic constructions from \cite{FarautKoranyi}.

Let us fix $Z_0 \in i\gc$ so that the eigenvalues of $\ad(Z_0)$ are $0,\pm 1$. Let $\gt \subset \gk$ be a Cartan subalgebra, so that $\gt$ is a Cartan subalgebra of $\gg$ as well. Note that $Z_0 \in i\gt$. For the Cartan subalgebra $\gt^\C$ of $\gg^\C$ we let $\Delta$ denote the corresponding root system. We denote by $\Delta_c$ and $\Delta_n$ the sets of compact and noncompact roots, respectively. Fix an order on the root system $\Delta$ with set of positive roots $\Delta^+$ so that
\[
    \Delta_n^+ = \Delta^+ \cap \Delta_n = \{\left. \alpha \,\right|\, \alpha(Z_0) = 1 \}.
\]
and denote $\Delta_n^- = -\Delta_n^+$. Consider the spaces
\[
    \gp_{\pm} = \sum_{\alpha \in \Delta_n^{\pm}} \gg_\alpha^\C
        = \gg^\C(\pm 1, \ad(Z_0)),
\]
which are Abelian subalgebras of $\gg^\C$ that are invariant under $\ad(\gk)$ and $\Ad(K)$. We denote by $P_+, K^\C, P_-$ the connected Lie subgroups of $G^\C$ whose Lie algebras are $\gp_+, \gk^\C, \gp_-$, respectively. Recall that the product map defines an injective holomorphic map $P_+ \times K^\C \times P_- \rightarrow G^\C$ whose image is an open dense subset of $G^\C$. In particular, for every $g$ in the image of the latter map we have a unique decomposition
\[
    g = p_+(g) k(g) p_-(g),
\]
where $p_\pm(g) \in P_\pm$ and $k(g) \in K^\C$.

We recall that $G \subset P_+ K^\C P_-$ and that $G \cap P_+ K^\C P_- = K$. In particular, we have the inclusion $G K^\C P_- \subset P_+ K^\C P_-$. This yields the Harish-Chandra realization of $G/K$ as a bounded symmetric domain in the complex vector space $\gp_+$ as the composition of the maps
\[
    G/K \simeq G K^\C P_-/K^\C P_- \hookrightarrow P_+ K^\C P_-/K^\C P_-
        \simeq P_+ \rightarrow \gp_+,
\]
which is explicitly given by the expression
\[
    gK \mapsto \log(p_+(g)).
\]
for every $g \in G$. Recall that the exponential map $\exp : \gp_+ \rightarrow P_+$ is a biholomorphism whose inverse is denoted by $\log : P_+ \rightarrow \gp_+$. In what follows, we will denote by $D$ the image in $\gp_+$ of the Harish-Chandra realization and consider the identification
\[
    G/K = D \subset \gp_+.
\]
Furthermore, through this identification we have a $G$-action, and so a $\widetilde{G}$-action as well, on $D$ by holomorphic diffeomorphisms given explicitly by the expression
\[
    g\cdot z = \log(p_+(g\exp(z)))
\]
for every $g \in G$ and $z \in D$.

Let $\mu$ be the normalized Lebesgue measure on $\gp_+$ so that $\mu(D) = 1$. We will denote by $\cH^2(D)$ the \textbf{(weightless) Bergman space} of $D$, which is defined as the (closed) subspace of $L^2(D,\mu)$ consisting of holomorphic functions on $D$. It is well known that $\cH^2(D)$ admits a reproducing kernel $k_D : D \times D \rightarrow \C$.

Recall that the genus of the domain $D$ is defined as
\[
    p = \frac{n + n_1}{r},
\]
where $n$ and $n_1$ are the complex dimensions of $D$ and of the maximal symmetric domain of tube type contained in $D$, respectively, and $r$ is the rank of $D$. Then, for every $\lambda > p - 1$ we define the measure
\[
    d\mu_\lambda(z) = c_\lambda k_D(z,z)^{1 - \frac{\lambda}{p}}dz,
\]
where $dz$ is the Lebesgue measure on $\gp_+$ and $c_\lambda$ is chosen so that $\mu_\lambda(D) = 1$. In particular, for $\lambda = p$ we obtain $\mu_p = \mu$, the normalized Lebesgue measure. The \textbf{weighted Bergman space} with weight $\lambda$ is denoted by $\cH^2_\lambda(D)$ and is defined as the (closed) subspace of $L^2(D, \mu_\lambda)$ consisting of holomorphic functions on $D$ (see \cite{Berezin, EnglisUpmeier}).

\begin{example}
    \label{ex:SU(n,m)}
    Let us consider the bounded symmetric domain $D_{n,m}^I$, where $n,m \geq 1$, associated to the simple Lie group $\SU(n,m)$. In this case we have
    \[
        Z_0 =
            \begin{pmatrix}
                \frac{m}{n+m} I_n & 0 \\
                0 & -\frac{n}{n+m} I_m
            \end{pmatrix}
    \]
    and also
    \begin{align*}
        \gp_+&=\left\{\left.
            \begin{pmatrix}
                0 & Z \\
                0 & 0
            \end{pmatrix}
                \,\right|\, Z\in M_{n\times m}(\C)\right\}\simeq M_{n\times m}(\C), \\
        \gp_-&=\left\{\left.
            \begin{pmatrix}
                0 & 0 \\
                W & 0
            \end{pmatrix}
                \,\right|\, W\in M_{m\times n}(\C)\right\}\simeq M_{m\times n}(\C), \\
        \gk&=\left\{\left.
            \begin{pmatrix}
                X & 0 \\
                0 & Y
            \end{pmatrix}
                \,\right|\,
                \begin{matrix}
                X\in \gu(n),\, Y \in \gu(m), \\
                \Tr X+\Tr Y =0
                \end{matrix}
                \right\} \simeq \gs(\gu(n)\times \gu(m)).
    \end{align*}
    Thus we obtain
    \begin{align*}
        P_+&=\left\{\left. \begin{pmatrix} I_n & Z\\ 0 & I_m\end{pmatrix}
            \,\right|\, Z\in M_{n\times m}(\C)\right\}, \\
        P_-&=\left\{\left. \begin{pmatrix} I_n & 0\\ W & I_m\end{pmatrix}
            \,\right|\, W\in M_{m\times n}(\C)\right\},
    \end{align*}
    and
    \begin{align*}
        K^\C &= \left\{\left.
            k(A,B)=
                \begin{pmatrix}
                    A& 0\\
                    0 & B
                \end{pmatrix}
            \, \right| \,
                \begin{matrix}
                    A\in \GL (n,\C),\, B\in \GL(m,\C) \\
                    \det A \det B=1
                \end{matrix}\right\} \\
        &\simeq S(\GL (n,\C)\times \GL (n,\C)) .
    \end{align*}
    From this, it is easy to verify that the decomposition of the elements in $P_+ K^\C P_-$ is given by the following expression
    \[
        \begin{pmatrix}
            A & B \\
            C & D
        \end{pmatrix} =
        \begin{pmatrix}
            I_n & BD^{-1} \\
            0 & I_m
        \end{pmatrix}
        \begin{pmatrix}
            A - BD^{-1}C & 0 \\
            0 & D
        \end{pmatrix}
        \begin{pmatrix}
            I_n & 0 \\
            D^{-1}C & 0
        \end{pmatrix}.
    \]
    Hence, for
    \[
        g =
        \begin{pmatrix}
            A & B \\
            C & D
        \end{pmatrix} \in \SU(n,m)
    \]
    we have
    \[
        g
        \begin{pmatrix}
            I_n & Z \\
            0 & I_m
        \end{pmatrix} =
        \begin{pmatrix}
            A & AZ + B \\
            C & CZ + D
        \end{pmatrix},
    \]
    and the previous decomposition yields
    \[
        p_+\left(
                g
        \begin{pmatrix}
            I_n & Z \\
            0 & I_m
        \end{pmatrix}
        \right) =
        \begin{pmatrix}
            I_n & (AZ + B)(CZ + D)^{-1} \\
            0 & I_m
        \end{pmatrix}.
    \]
    We conclude that the $\SU(n,m)$-action on $D_{n,m}^I$ is given by
    \begin{align*}
        \SU(n,m) \times D_{n,m}^I &\rightarrow D_{n,m}^I \\
        \begin{pmatrix}
            A & B \\
            C & D
        \end{pmatrix} \cdot Z &= (AZ + B)(CZ + D)^{-1}.
    \end{align*}
    In particular, the $K^\C$-action on $\gp_+$ is given by the expression
        \[
            \begin{pmatrix}
                A & 0 \\
                0 & B
            \end{pmatrix}
            \cdot Z=AZB^{-1}  .
        \]
\end{example}

Back to the case of a general bounded symmetric domain $D$, since the measure $\mu_\lambda$ is positive on nonempty open sets, by the same proof as in the weightless case, for every $z \in D$ the linear functional $f \mapsto f(z)$ is continuous on $\cH^2_\lambda(D)$. This ensures the existence of a reproducing kernel $k_\lambda : D \times D \rightarrow \C$ for $\cH^2_\lambda(D)$ known as the Bergman kernel (see \cite{HelgasonBook}). Moreover, the kernel $k_\lambda(z,w)$ is holomorphic in $z$ and anti-holomorphic in $w$. Hence, the \textbf{Bergman projection} defined by
\begin{align*}
    B_\lambda : L^2(D, \mu_\lambda) &\rightarrow \cH^2_\lambda(D) \\
        B_\lambda(f)(z) &= \int_D f(w) k_\lambda(z,w) d\mu_\lambda(w),
\end{align*}
fixes any $f \in \cH^2_\lambda(D)$ and vanishes on any $f \in \cH^2_\lambda(D)^\perp$. In particular, the Bergman projection is precisely the orthogonal projection onto $\cH^2_\lambda(D)$.

By Section~3 of \cite{FarautKoranyi}, for every $\lambda > p - 1$ the weighted Bergman space $\cH^2_\lambda(D)$ admits a unitary $\widetilde{G}$-action given by
\[
    f(z) \mapsto j(g^{-1},z)^\frac{\lambda}{p} f(g^{-1}z)
\]
where $g \in \widetilde{G}$ and $j(g,z)$ denotes the complex Jacobian at $z$ of the holomorphic transformation of $D$ defined by $g$. This defines the \textbf{holomorphic discrete series of $\widetilde{G}$} for $\lambda > p - 1$. Note that we need to consider an action of the group $\widetilde{G}$ for the expression $j(g,z)^\frac{\lambda}{p}$ to be a well-defined holomorphic function for arbitrary $\lambda > p - 1$.

Hence, we have the following result.

\begin{proposition} \label{prop:Bergman-as-discrete-series}
    The holomorphic discrete series of $\widetilde{G}$ for $\lambda > p - 1$ has representation space given by $\cH^2_\lambda(D)$, with the action
    \[
        \pi_\lambda(g)(f)(z) = j(g^{-1},z)^\frac{\lambda}{p} f(g^{-1}z)
    \]
    for $f \in \cH^2_\lambda(D)$, $z \in D$ and $g \in \widetilde{G}$. Furthermore, $\pi_\lambda$ is an irreducible unitary representation of $\widetilde{G}$ for every $\lambda > p - 1$.
\end{proposition}

\section{Intertwining Toeplitz operators}
\label{sec:interToeplitz}
Let $H$ be a Lie group and $\pi : H \rightarrow \cL(\cH)$ a unitary representation on a Hilbert space $\cH$. The set of operators on $\cH$ that intertwine the representation $\pi$ is denoted by $\End_H(\cH)$. More precisely, a bounded operator $T$ on $\cH$ belongs to $\End_H(\cH)$ if and only if we have
\[
    \pi(h) \circ T = T \circ \pi(h),
\]
for every $h \in H$. We observe that $\End_H(\cH)$ is a subalgebra of $\cL(\cH)$.

With the notation from the previous section, let $D$ be the Harish-Chandra realization of the Hermitian symmetric space associated to $G$. For every function $\varphi \in L^\infty(D)$ we define the \textbf{Toeplitz operator with symbol $\varphi$} corresponding to the weight $\lambda > p - 1$ as the bounded linear operator given by
\begin{align*}
    T^{(\lambda)}_\varphi : \cH^2_\lambda(D) &\rightarrow \cH^2_\lambda(D) \\
        T^{(\lambda)}_\varphi &= B_\lambda \circ M_\varphi,
\end{align*}
where, as usual, $M_\varphi$ denotes the multiplier operator on $L^2(D, \mu_\lambda)$ given by $\varphi$. In particular, we have
\[
    T^{(\lambda)}_\varphi(f)(z) = \int_D \varphi(w) f(w) k_\lambda(z,w) d\mu_\lambda(w),
\]
for every $f \in \cH^2_\lambda(D)$ and $z \in D$. It is interesting to note that
\[
    T^{(\lambda)}_\varphi = M_\varphi
\]
if $\varphi \in L^\infty(D)$ is holomorphic.

The following well known result can be found in \cite{VasilevskiBook}.

\begin{lemma}
\label{lem:injectiveToeplitz}
    For every $\lambda > p - 1$, the linear map
    \begin{align*}
        L^\infty(D) &\rightarrow \cL(\cH^2_\lambda(D)) \\
            \varphi &\mapsto T^{(\lambda)}_\varphi
    \end{align*}
    is injective.
\end{lemma}
\begin{proof}
Suppose that $\varphi\in L^\infty(D)$ and that $T^{(\lambda)}_\varphi = 0$ for some $\lambda$.  Then for each pair of holomorphic monomials $z^\alpha$ and $z^\beta$, we have that
\begin{align*}
   0 &=  \langle T^{(\lambda)}_\varphi z^\alpha, z^\beta \rangle \\
        &= \langle B_\lambda M_\varphi z^\alpha, z^\beta \rangle \\
        &= \langle M_\varphi z^\alpha, z^\beta \rangle \\
        &=  \langle \varphi, \overline{z^\alpha} z^\beta \rangle,
\end{align*}
from which it follows that $\varphi$ is orthogonal to all real polynomials on $D$. Since the real polynomials are dense in $L^2(D,\mu_\lambda)$, it follows that $\varphi=0$.
\end{proof}

For every function $\varphi$ defined on $D$ and $h \in G$ we denote by $\varphi_h$ the function given by $\varphi_h(z) = \varphi(h^{-1}z)$. This clearly defines both a $G$-action and a $\widetilde{G}$-action on the space $L^\infty(D)$. The following result relates this action to $\pi_\lambda$ and Toeplitz operators.

\begin{lemma}
\label{lem:intertwiningToeplitz}
    If $\lambda > p - 1$, then we have
    \[
        \pi_\lambda(h) \circ T^{(\lambda)}_\varphi = T^{(\lambda)}_{\varphi_h} \circ \pi_\lambda(h),
    \]
    for every $\varphi \in L^\infty(D)$ and $h \in \widetilde{G}$.
\end{lemma}

\begin{proof}
First we note that $\pi_\lambda(h) \circ M_\varphi = M_{\varphi_h} \circ \pi_\lambda(h)$, for every $\varphi \in L^\infty(D)$ and $h \in \widetilde{G}$. In fact, for any $f\in \cH^2_\lambda(D)$, $h \in \widetilde{G}$ and $z \in D$, we have
\begin{align*}
    \pi_\lambda(h) M_\varphi f(z) & = \varphi(h^{-1}z) f(h^{-1}z) j(h^{-1},z)^\frac{\lambda}{p} \\
                   & = M_{\varphi_h} \pi_\lambda(h)f(z).
\end{align*}
Now using the fact that $B_\lambda$ is the orthogonal projection onto $\cH^2_\lambda(D)$ we have for every $f, g \in \cH^2_\lambda(D)$ and $h \in \widetilde{G}$
\begin{align*}
    \langle \pi_\lambda(h) \circ T^{(\lambda)}_\varphi f,g \rangle
        & = \langle B_\lambda \circ M_\varphi f, \pi_\lambda(h^{-1}) g \rangle \\
        & = \langle M_\varphi f, \pi_\lambda(h^{-1}) g \rangle \\
        & = \langle M_{\varphi_h} \circ \pi_\lambda(h) f, g \rangle \\
        & = \langle T^{(\lambda)}_{\varphi_h} \circ \pi_\lambda(h) f, g \rangle.
\end{align*}
Thus proving our claim.
\end{proof}

As a consequence, we obtain the following characterization of Toeplitz operators that intertwine the action of a subgroup of $\widetilde{G}$.

\begin{corollary}
\label{cor:intertwiningToeplitz}
    Suppose that $\widetilde{H}$ is a subgroup of $\widetilde{G}$. Then, the following conditions are equivalent for every $\varphi \in L^\infty(D)$.
    \begin{enumerate}
        \item The symbol $\varphi$ is $\widetilde{H}$-invariant: $\varphi(hz) = \varphi(z)$ for every $h \in \widetilde{H}$ and a.e.~$z \in D$.
        \item The Toeplitz operator $T^{(\lambda)}_\varphi$ intertwines the restriction $\pi_\lambda|_{\widetilde{H}}$ of $\pi_\lambda$ to $\widetilde{H}$: $\pi_\lambda(h) \circ T^{(\lambda)}_\varphi = T^{(\lambda)}_\varphi \circ \pi_\lambda(h)$ for every $h \in \widetilde{H}$.
    \end{enumerate}
\end{corollary}
\begin{proof}
    By Lemma~\ref{lem:intertwiningToeplitz}, we have that $T_\varphi$ is an intertwining operator for $\pi_\lambda|_{\widetilde{H}}$ if and only if $T_\varphi = T_{\varphi_h}$ for all $h\in \widetilde{H}$. By Lemma~\ref{lem:injectiveToeplitz}, this is equivalent to $\varphi=\varphi_h$ for all $h\in \widetilde{H}$.
\end{proof}

\section{Multiplicity-free restrictions}
\label{sec:multFree}
Let $H$ be a Lie group of type I in the sense of von Neumann algebras and let $\pi : H \rightarrow \cL(\cH)$ be a unitary representation of $H$ on a Hilbert space $\cH$. By Mautner's Theorem we have a unique decomposition
\[
    \pi \simeq \int_{\widehat{H}} m_\pi(\rho)\; \rho\; d\nu(\rho)
\]
where $d\nu(\rho)$ is some Borel measure on the space $\widehat{H}$ of equivalence classes of irreducible unitary representations of $H$, and $m_\pi : \widehat{H} \rightarrow \N \cup \{\infty\}$ is the multiplicity function which is $\nu$-a.e.~unique. Given this setup, we say that the representation $\pi$ is \textbf{multiplicity-free} if the function $m_\pi$ is a characteristic function, in other words, $\nu$-a.e.~it is either $0$ or $1$.

We have the following alternative characterization of a multiplicity-free representation.

\begin{proposition}
    \label{prop:typeImultiplicityFree}
    Let $H$ be a Lie group of type I in the sense of von Neumann algebras and let $\pi : H \rightarrow \cL(\cH)$ be a unitary representation of $H$ on a Hilbert space $\cH$. Then, $\pi$ is multiplicity-free if and only if $\End_H(\cH)$ is commutative.
\end{proposition}

The next result establishes the existence of commutative $C^*$-algebras generated by Toeplitz operators from multiplicity-free restrictions of the holomorphic discrete series representations of $\widetilde{G}$. From now on, for $\cA \subset L^\infty(D)$ a set of bounded symbols we will denote by $\cT^{(\lambda)}(\cA)$ the subalgebra of $\cL(\cH^2_\lambda(D))$ generated by the set of Toeplitz operators $\{ T^{(\lambda)}_\varphi : \varphi \in \cA \}$. Note that if $\cA$ is invariant under conjugation, then $\cT^{(\lambda)}(\cA)$ is in fact a $C^*$-algebra. Here, and the rest of this work, for $H$ a subgroup of $G$, we will denote by $\widetilde{H}$ the inverse image of $H$ with respect to the universal covering map $\widetilde{G} \rightarrow G$.

\begin{theorem}
    \label{thm:commToeplitz-multiplicityFree}
    Let $H$ be a closed subgroup of $G$ and let us denote by $\cA^H$ the subspace of $L^\infty(D)$ that consists of the $H$-invariant bounded symbols on $D$. If for some $\lambda > p - 1$ the algebra $\End_{\widetilde{H}}(\cH^2_\lambda(D))$ is commutative, then $\cT^{(\lambda)}(\cA^H)$ is a commutative $C^*$-algebra. In particular, the result holds if $H$ is a type I group, in the sense of von Neumann algebras, and the restriction $\pi_\lambda|_{\widetilde{H}}$ is multiplicity-free.
\end{theorem}
\begin{proof}
    Since $\{ T^{(\lambda)}_\varphi : \varphi \in \cA^H \}$ generates the $C^*$-algebra $\cT^{(\lambda)}(\cA^H)$ and $\cA^H = \cA^{\widetilde{H}}$, it is enough to prove that $T^{(\lambda)}_\varphi$ and $T^{(\lambda)}_\psi$ commute for every pair of symbols $\varphi, \psi \in \cA^{\widetilde{H}}$. By Corollary~\ref{cor:intertwiningToeplitz}, the $\widetilde{H}$-invariance of any given $\varphi, \psi \in \cA^{\widetilde{H}}$ implies that the corresponding Toeplitz operators $T^{(\lambda)}_\varphi$ and $T^{(\lambda)}_\psi$ both belong to $\End_{\widetilde{H}}(\cH^2_\lambda(D))$ which is commutative by hypothesis.
\end{proof}

\section{Commuting Toeplitz operators: symmetric pair case}
\label{sec:commToeplitz-symmPairs}
There are a number of general results showing that for the holomorphic discrete series representations $\pi_\lambda$ the restriction from $\widetilde{G}$ to suitable closed subgroups $H$ provides multiplicity-free representations $\pi_\lambda|_H$. The case we consider in this section is that of subgroups defining a symmetric pair.

We recall that a \textbf{symmetric pair} is given by a pair of Lie groups $(H_1, H_2)$, so that $H_1$ is connected semisimple, and an involutive automorphism $\tau$ of $H_1$ so that $H_2$ is an open subgroup of $H_1^\tau$, where the latter denotes the fixed point set of $\tau$.

For $G$ as before, let $\theta$ be the Cartan involution associated to the decomposition $\gg = \gk + \gp$. Let us consider a symmetric pair $(G,H)$ with involution $\tau$. Hence, we can always replace $\tau$ by a conjugate to assume that $\tau$ and $\theta$ commute. In particular, $\tau(\gk) = \gk$ and since $\dim(\gc) = 1$ we also have $\tau(Z_0) = \pm Z_0$. We say that the $(G,H)$ is a \textbf{holomorphic symmetric pair} when $\tau(Z_0) = Z_0$ holds, and that it is an \textbf{anti-holomorphic symmetric pair} when $\tau(Z_0) = -Z_0$ is satisfied.

Note that for any symmetric pair $(G,H)$, its involution $\tau$ defines an involution of the symmetric space $G/K$ that we denote by the same symbol.

In the holomorphic case, we necessarily have $\tau(\gp_+) = \gp_+$ and $\tau(\gp_-) = \gp_-$, so that the restriction of $\tau$ to $D \subset \gp_+$ yields a holomorphic involution of $D$. Also, the fixed point set of $\tau$ in $D$ satisfies
\[
    D^\tau = H/H\cap K,
\]
with respect to the Harish-Chandra realization of $D$. Furthermore, the space $D^\tau$ is a bounded symmetric domain in $\gp_+ \cap \gh^\C$.

Now suppose that $(G,H)$ is an anti-holomorphic symmetric pair, and let $\sigma$ denote the conjugation of $\gg^\C$ with respect to $\gg$ as well as the induced conjugation on $G^\C$. Consider the anti-linear involution $\eta = \sigma \circ \tau = \tau \circ \sigma$. Since $Z_0 \in i \gc \subset i\gg$ we have $\eta(Z_0) = Z_0$, thus implying that the induced involution on $G^\C$ preserves the decomposition $P_+ K^\C P_-$. More precisely, we have
\[
    \eta(g) = \exp(\eta(Z)) \eta(k) \exp(\eta(W))
\]
for every $g \in G$ given by $g = \exp(Z) k \exp(W)$ with $Z \in \gp_+$, $k \in K^\C$ and $W \in \gp_-$. In particular, $\eta$ restricts to an anti-linear involution on $\gp_+$. On the other hand, since $\tau|_G = \eta|_G$ we conclude that $\tau$ induces an involution of $G/K$, that we will denote by the same symbol, so that with respect to the Harish-Chandra identification $G/K = D$ we have
\[
    \left(G/K\right)^\tau = D \cap \gp_+^\eta.
\]
Hence, if we denote $D^\tau = D \cap \gp_+^\eta$, then $D^\tau$ is a totally real submanifold of $D$. In this case we have
\[
    D^\tau = H/H\cap K,
\]
with respect to the Harish-Chandra identification for $D$. Also, $D^\tau$ is a Riemannian symmetric space and a real bounded symmetric domain with the restriction of the Bergman metric in $D$. We refer to \cite{HilgertOlafsson} for details on the structure of the domain $D^\tau$.

Table I lists the simple Lie algebras corresponding to bounded symmetric domains and the subalgebras that define non-Riemannian symmetric pairs. We present in separate columns the cases for $D^\tau$ complex and totally real submanifold of $D$. The classification of $D^\tau$ complex is taken from \cite[Table 3.4.1,p.65]{Kobayashi} (note that $\so (10)\times \so (2)\subset \gf_{6(-14)}$ and $\gf_{6(-78)}\times \so (2)\subset \gf_{7(-25)}$ should be removed from that table). The symmetric Lie algebras corresponding to totally real $D^\tau$ is taken from \cite[p.89]{HilgertOlafsson} (see also \cite{Ol1991}).

In \cite{Kobayashi} it is presented a very general theory showing that the restriction of highest weight representations are multiplicity-free. Based on our previous results, this allows to obtain the following conclusion.

\begin{theorem}
    \label{thm:commToeplitz-symmPairs}
    If $(G,H)$ is a symmetric pair, then the $C^*$-algebra $\cT^{(\lambda)}(\cA^H)$ is commutative on the weighted Bergman space $\cH^2_\lambda(D)$ for every $\lambda > p - 1$. In particular, the conclusion holds for $(G,K)$ and for every symmetric pair $(G,H)$ so that $(\gg, \gh)$ is a pair listed in Table I.
\end{theorem}
\begin{proof}
    Theorem~A from \cite{Kobayashi} implies that, since $(G,H)$ is a symmetric pair, the restriction to $\widetilde{H}$ of any irreducible unitary highest weight representation of $\widetilde{G}$ is multiplicity-free. From our discussion in Section~\ref{sec:BergmanSpaces} it follows that the representations of $\widetilde{G}$ on the Bergman spaces $\cH^2_\lambda(D)$ belong to the holomorphic discrete series and so are irreducible unitary highest weight representations. Hence, the result now follows from Theorem~\ref{thm:commToeplitz-multiplicityFree}.
\end{proof}

\begin{table}
\begin{center}
\begin{tabular}{|c|c|c|}
\hline
\multicolumn{3}{|c|}{}\\
\multicolumn{3}{|c|}{TABLE I}\\
\multicolumn{3}{|c|}{}\\
\hline
\hline
$\gg$ &
$D^\tau$ complex  &  $D^\tau$ totally real \\
\hline
$\su (n,m)$ & $\gs(\gu (i,j)\times \gu (n-i,m-j))$& $\so (n,m)$\\
$\su (2n,2m)$ & &$\gsp (n,m)$\\
\hline
$\su (n,n)$ & $\so^*(2n)$ & $\gsl (n,\C)\times \R$ \\
& $\gsp (n,\R)$ & \\
\hline
$\so^*(2n)$ &$\so^*(2i)\times \so^*(2(n-i))$ &$\so (n,\C)$ \\
&$ \gu (i,n-i)$& \\
$\so^*(4n)$& & $\su^*(2n)\times\R$\\
\hline
$\so (2,n)$ & $\so (2,i)\times \so (n-i)$& $\so (1,i)\times \so (1,n-i)$ \\
$\so(2,2n)$ & $\gu (1,n)$ & \\
\hline
$\gsp (n,\R)$ &$ \gu (i,n-i)$ & $\mathfrak{gl}(n,\R)$ \\
& $\gsp (i,\R)\times \gsp (n-i,\R)$ & \\
$\gsp (2n,\R)$ & & $\gsp (n,\C)$\\
\hline
$\ge_{6(-14)}$ & $\so^* (10)\times \so (2)$ &$\mathfrak{f}_{4(-20)}$ \\
&$\so (8,2)\times \so (2)$ & $\gsp (2,2)$\\
& $\su(5,1)\times \gsl (2,\R)$ & \\
&$\su(4,2)\times \su (2)$ & \\
\hline
$\ge_{7(-25)}$& $\ge_{6(-14)}\times \so(2)$ & $\ge_{6(-26)}\times \so (1,1)$ \\
& $\so (10,2)\times \gsl (2,\R)$ & $\su^*(8)$ \\
& $\so^*(12)\times \su (2)$&\\
& $\su (6,2)$& \\
\hline
\hline
\end{tabular}
\end{center}
\end{table}

Since $(G,K)$ is a (Riemannian) symmetric pair, we obtain the following result as a particular case.

\begin{theorem}
    \label{thm:commToeplitz-maxCompact}
    Let $\cA^K$ be the set of bounded symbols that are $K$-invariant. Then, the $C^*$-algebra $\cT^{(\lambda)}(\cA^K)$ is commutative on the weighted Bergman space $\cH^2_\lambda(D)$ for every $\lambda > p - 1$.
\end{theorem}

\begin{remark}
    Theorem~\ref{thm:commToeplitz-maxCompact} provides, to the best of our knowledge, the first existence result of large commutative $C^*$-algebras generated by Toeplitz operators on bounded symmetric domains other than the unit ball. Furthermore, since there are many symmetric pairs $(G,H)$ with noncompact $H$ (see Table I) the list of commutative $C^*$-algebras provided by Theorem~\ref{thm:commToeplitz-symmPairs} is much larger than the one obtained from Theorem~\ref{thm:commToeplitz-maxCompact}.
\end{remark}

Let us now further consider the case where $(G,H)$ is an anti-holomorphic symmetric pair. The relevant results for this case are extensively studied in \cite{OlafssonOrsted} (see also \cite{Ol2000}). Such results allow us to conclude multiplicity-free restrictions from which the next result follows.

\begin{theorem}
    \label{thm:commToeplitz-antihol}
    If $(G,H)$ is an anti-holomorphic symmetric pair, then the $C^*$-algebra $\cT^{(\lambda)}(\cA^H)$ is commutative on the weighted Bergman space $\cH^2_\lambda(D)$ for every $\lambda > p - 1$. Furthermore, for every such $\lambda$ there is a unitary isomorphism $U_\lambda : \cH^2_\lambda(D) \rightarrow L^2(D^\tau, \widehat{\mu})$ so that the following is satisfied.
    \begin{enumerate}
        \item The measure $\widehat{\mu}$ is an $\widetilde{H}$-invariant measure in $D^\tau$ independent of $\lambda$.
        \item The map $U_\lambda$ is $\widetilde{H}$-intertwining for the action on $L^2(D^\tau, \widehat{\mu})$ given by $hf(x) = f(h^{-1}x)$.
        \item For every $\varphi \in \cA^H$ the operator $U_\lambda \circ T^{(\lambda)}_\varphi \circ U_\lambda^{-1}$ is a multiplier operator up to the Helgason-Fourier transform of $L^2(D^\tau, \widehat{\mu})$.
    \end{enumerate}
\end{theorem}
\begin{proof}
    The commutativity of $\cT^{(\lambda)}(\cA^H)$ has already been established in Theorem~\ref{thm:commToeplitz-symmPairs}, so it remains to prove the last three claims. This will be based on the results from \cite{Ol2000, OlafssonOrsted}.

    In \cite{OlafssonOrsted} it is constructed an $\widetilde{H}$-invariant measure on $D^\tau$ and a unitary $\widetilde{H}$-intertwining isomorphism $U_\lambda : \cH^2_\lambda(D) \rightarrow L^2(D^\tau, \widehat{\mu})$ for which conditions (1) and (2) hold. On the other hand, $D^\tau$ is a reductive symmetric space so that the set of transformations given by $H$ contains the connected component of its isometry group. Also, $\widehat{\mu}$ is an $H$-invariant volume form on $D^\tau$. Hence, the decomposition of $L^2(D^\tau, \widehat{\mu})$ as a direct integral of equivalence classes of irreducible unitary representations of $\widetilde{H}$ is multiplicity-free. This follows from \cite[Eq.(21), p.118]{Helgason1970} and \cite{HarishChandraSpherical1,HarishChandraSpherical2}. Moreover, it is also known that for $T \in \End_H(L^2(D^\tau, \widehat{\mu}))$ the operator $\cF \circ T \circ \cF^{-1}$ is a multiplier operator, where $\cF$ is the Helgason-Fourier transform. Since we have
    \begin{align*}
        U_\lambda \End_{\widetilde{H}}(\cH^2_\lambda(D)) U_\lambda^{-1}
            &\subset \End_H(L^2(D^\tau, \widehat{\mu})), \\
        \cT^{(\lambda)}(\cA^H) &\subset \End_{\widetilde{H}}(\cH^2_\lambda(D)),
    \end{align*}
    and so claim (3) now follows.
\end{proof}

We now discuss some particular cases of Theorem~\ref{thm:commToeplitz-symmPairs} related to some of the already known commutative $C^*$-algebras generated by Toeplitz operators.

\begin{example}
    The unit ball in $\C^n$ can be realized as $\B^n = \SU(n,1)/\U(n)$. In particular, the pair $(\SU(n,1), \U(n))$ is symmetric with $\U(n)$ a maximal compact subgroup of $\SU(n,1)$. Hence, Theorem~\ref{thm:commToeplitz-maxCompact} implies that, for every $\lambda > n$, the $C^*$-algebra $\cT^{(\lambda)}(\cA^{\U(n)})$ is commutative on the weighted Bergman space $\cA^2_\lambda(\B^n)$.

    From the definitions it follows that $\cA^{\U(n)}$ consists of radial symbols, in other words those that are defined as the elements $a \in L^\infty(\B^n)$ that depend only on $|z|$ for $z \in \B^n$ or symbolically that satisfy $a(z) = a(|z|)$. This recovers one of the main results from \cite{GKVRadial}: the $C^*$-algebra generated by Toeplitz operators with radial symbols is commutative.
\end{example}

\begin{example}
    By Table~I, $(\SU(m,n), \SO_0(m,n))$ is a symmetric pair. Hence, Theorem~\ref{thm:commToeplitz-symmPairs} implies that $\cT^{(\lambda)}(\cA^{\SO_0(m,n)})$ is a commutative $C^*$-algebra on the weighted Bergman space $\cH^2_\lambda(D_{n,m}^I)$ for every $\lambda > m + n - 1$.

    As a particular case, by taking $m = 1$ we conclude that $\cT^{(\lambda)}(\cA^{\SO_0(n,1)})$ is a commutative $C^*$-algebra on the weighted Bergman space $\cA^2_\lambda(\B^n)$ for every $\lambda > n$. We note that the $\SO_0(n,1)$-orbit through the origin in $\B^n$ is precisely $M = \B^n \cap \R^n$, which is isometric to the real $n$-dimensional hyperbolic space with the induced metric from $\B^n$. Also, $M$ is a Lagrangian totally geodesic submanifold of $\B^n$.
\end{example}

As a consequence of the previous example we obtain the following result.

\begin{theorem}\label{thm:commToeplitzBn-nonflat}
    For $n \geq 2$, there exists a closed non-Abelian subgroup $H$ of $\SU(n,1)$ that satisfies the following properties.
    \begin{enumerate}
        \item For every $\lambda > n$, the $C^*$-algebra $\cT^{(\lambda)}(\cA^H)$ is commutative on the weighted Bergman space $\cA^2_\lambda(\B^n)$.
        \item There is an $H$-orbit $M$ which is $n$-dimensional, non-flat, Lagrangian and totally geodesic.
    \end{enumerate}
    In particular, the sets of symbols $\cA$ that define commutative $C^*$-algebras $\cT^{(\lambda)}(\cA)$, for every $\lambda > n$, are not exhausted by taking $\cA = \cA^H$ where $H$ is a connected maximal Abelian subgroup of $\SU(n,1)$, even if we require (2) to hold.
\end{theorem}

For our next result we assume that $D$ is a tube-type domain whose tube domain realization is given by $T(\Omega) = V + i\Omega$, where $\Omega \subset V$ is a symmetric cone and $V$ is the corresponding simple Euclidean Jordan algebra. Following \cite{ADO} we recall that for every $\lambda > p - 1$ there is a measure $\widehat{\mu}_\lambda$ in $\Omega$, a unitary representation of $\widetilde{G}$ on $L^2(\Omega, \widehat{\mu}_\lambda)$ and a unitary isomorphism
\[
    U_\lambda : \cH^2_\lambda(D) \rightarrow L^2(\Omega, \widehat{\mu}_\lambda)
\]
which is $\widetilde{G}$-intertwining. On the other hand, there is a biholomorphic and isometric action given by
\begin{align*}
    V \times T(\Omega) &\rightarrow T(\Omega), \\
        (v, x) &\mapsto v + x.
\end{align*}
Since the $\widetilde{G}$-action on $D$ realizes the holomorphic isometries of $D$, we obtain a subgroup $N^+$ of $\widetilde{G}$ isomorphic to the vector group $V$ such that the $N^+$-action on $\Omega$ realizes the previous $V$-action. Furthermore, \cite[Theorem 3.4]{ADO} implies that the $N^+$-action on $L^2(\Omega, \widehat{\mu})$ is given by characters of $N^+ \simeq V$ and so it is multiplicity-free.

Hence, Theorem~\ref{thm:commToeplitz-multiplicityFree} implies the following generalization of one of the main results from \cite{VasilevskiTube} where only the weightless case is considered.

\begin{theorem}
    \label{thm:commToeplitz-tube}
    The $C^*$-algebra $\cT^{(\lambda)}(\cA^{N^+})$ is commutative on the weighted Bergman space $\cH^2_\lambda(D)$ for every $\lambda > p - 1$. Furthermore, up to the unitary isomorphism $U_\lambda : \cH^2_\lambda(D) \rightarrow L^2(\Omega, \widehat{\mu}_\lambda)$ the elements of $\cT^{(\lambda)}(\cA^{N^+})$ are multiplier operators for every $\lambda > p - 1$.
\end{theorem}

\begin{example}
    As a particular case of the unit ball example, the unit disk $\D \subset \C$ is given by the Riemannian symmetric pair $(SU(1,1), \U(1))$. It is well known that every connected $1$-dimensional subgroup of $\SU(1,1)$ is conjugate to one of the following subgroups
    \begin{align*}
        \U(1) &= \left\{
            \begin{pmatrix}
                e^{it} & 0 \\
                0 & e^{-it}
            \end{pmatrix}
            : t \in \R \right\}, \\
        \SO_0(1,1) &= \left\{
            \begin{pmatrix}
                \cosh(t) & \sinh(t) \\
                \sinh(t) & \cosh(t)
            \end{pmatrix} : t \in \R \right\}, \\
        N &= \left\{
            \begin{pmatrix}
                1 + it & -it \\
                it & 1 - it
            \end{pmatrix} : t \in \R \right \}.
    \end{align*}
    The orbits of these groups are all given by circle arcs, and considering the description of such arcs for each case we obtain the following description of the corresponding invariant symbols.
    \begin{description}
        \item[Elliptic] The elements of $\cA^{\U(1)}$ are the radial bounded symbols as considered in the more general case of the unit ball. These symbols are also called
            elliptic in \cite{GQV}.
        \item[Hyperbolic] The elements of $\cA^{\SO_0(1,1)}$ are the bounded symbols that are constant on the circle arcs that connect $1$ and $-1$. These are precisely the hyperbolic bounded symbols in the notation of \cite{GQV}.
        \item[Parabolic] The elements of $\cA^N$ are the bounded symbols that are constant on the horocycles centered at $1$. These are precisely the parabolic bounded symbols in the notation of \cite{GQV}
    \end{description}

    Since the pairs $(\SU(1,1), \U(1))$ and $(\SU(1,1), \SO_0(1,1))$ are both symmetric, it follows from Theorem~\ref{thm:commToeplitz-symmPairs} that, for every $\lambda > 1$, the $C^*$-algebras $\cT^{(\lambda)}(\cA^{\U(1)})$ and $\cT^{(\lambda)}(\cA^{\SO_0(1,1)})$ are commutative on the weighted Bergman space $\cA^2_\lambda(\D)$. This recovers Theorem~2.1 from \cite{GQV} for elliptic and hyperbolic bounded symbols (see also \cite{GKVHyperbolic}).

    On the other hand, we note that $(\SU(1,1), N)$ is not a symmetric pair, and Theorem~\ref{thm:commToeplitz-symmPairs} cannot be applied in this case. In the half-plane realization of $\D$, the $N$-action is precisely the one considered by Theorem~\ref{thm:commToeplitz-tube}. Hence, the $C^*$-algebra $\cT^{(\lambda)}(\cA^N)$ is commutative on the weighted Bergman space $\cA^2_\lambda(\D)$ for every $\lambda > 1$. Thus, we fully recover all cases considered in Theorem~2.1 from \cite{GQV} (see also \cite{GKVParabolic}).
\end{example}

\section{Commuting Toeplitz operators: compact group case}
\label{sec:commToeplitz-compact}
In this section $H$ will denote a compact subgroup of $G$ and it will be endowed with a Haar measure whose total volume is $1$. Then, for every symbol $\varphi \in L^\infty(D)$ we define
\[
    \widehat{\varphi}(z) = \int_H \varphi_h(z) dh = \int_H \varphi(h^{-1}z) dh
\]
where $z \in D$. Hence, $\widehat{\varphi}$ is both $H$-invariant and $\widetilde{H}$-invariant. A particular case is given by $H = K$, the maximal compact subgroup of $G$ introduced in Section~\ref{sec:BergmanSpaces}.

Similarly, for every $T \in \cL(\cH^2_\lambda(D))$ we define a bounded linear operator obtained by averaging over $H$. First, we note that the relevant expression to consider is
\[
    \pi_\lambda(h) \circ T \circ \pi_\lambda(h)^{-1}
\]
for $h \in \widetilde{H}$, where now $\widetilde{H}$ is not necessarily a compact group. However, $Z(\widetilde{H}) = \widetilde{H} \cap Z(\widetilde{G}) \subset Z(\widetilde{G})$ acts through $\pi_\lambda$ as multiplication by elements of $\T$. In particular, for every $h \in \widetilde{H}$ and $h_0 \in Z(\widetilde{H})$ we have
\[
    \pi_\lambda(h) \circ T \circ \pi_\lambda(h)^{-1}
    = \pi_\lambda(hh_0) \circ T \circ \pi_\lambda(hh_0)^{-1}.
\]

On the other hand, the restriction of the universal covering map $\widetilde{G} \rightarrow G$ yields a map $\widetilde{H} \rightarrow H$ which is a local isomorphism and so a covering map. It follows that there is a discrete subgroup $Z \subset Z(\widetilde{H})$ so that $H = \widetilde{H}/Z$. Hence, the above remarks show that the assignment
\[
    h \mapsto \pi_\lambda(h) \circ T \circ \pi_\lambda(h)^{-1}
\]
defines a function that descends to $H$. Without loss of generality, we will denote the latter by the same expression. As a consequence, for every operator $T \in \cL(\cH^2_\lambda(D))$ there is a unique operator $\widehat{T} \in \cL(\cH^2_\lambda(D))$ defined by requiring the identity
\[
    \langle \widehat{T} f, g \rangle = \int_H \langle\pi_\lambda(h)\circ T\circ \pi_\lambda(h)^{-1} f, g \rangle dh
\]
to hold for every $f,g \in \cH^2_\lambda(D)$. We note that $\widehat{T}$ belongs to $\End_{\widetilde{H}}(\cH^2_\lambda(D))$.

\begin{lemma}
\label{lem:averagingToeplitz}
If $H \subset G$ is a compact subgroup, then we have
\[
    \widehat{T_\varphi^{(\lambda)}} = T^{(\lambda)}_{\widehat{\varphi}}
\]
for every $\varphi \in L^\infty(D)$ and every $\lambda > p - 1$.
\end{lemma}
\begin{proof}
For each $f,g\in \cH^2_\lambda(D)$, we have that
\begin{align*}
  \langle \widehat{T^{(\lambda)}_\varphi} f, g \rangle
    & = \int_H \langle \pi_\lambda(h)\circ T_\varphi^{(\lambda)} \circ \pi_\lambda(h)^{-1} f, g \rangle dh \\
    & = \int_H \langle T_{\varphi_h}^{(\lambda)} f,g \rangle dh \\
    & = \int_H \int_D \varphi_h(z) f(z)\overline{g(z)} d\mu_\lambda(z) dh \\
    & = \int_D \left( \int_H \varphi_h(z)dh \right) f(z)\overline{g(z)} d\mu_\lambda(z) \\
    & = \langle T^{(\lambda)}_{\widehat{\varphi}} f,g\rangle.
\end{align*}
\end{proof}

The following result is a generalization of Theorem~2 from \cite{Englis} to intertwining operators, and so it is interesting by its own right.

\begin{proposition}
\label{prop:DenseInvariantToeplitz}
Suppose that $H$ is a compact subgroup of $G$.  Then the space $\cT^{(\lambda)}(\cA^H)$ of Toeplitz operators on $\cH^2_\lambda(D)$ with $H$-invariant-symbols in $L^\infty(D)$ is dense in the space $\End_{\widetilde{H}}(\cH^2_\lambda(D))$ of $\widetilde{H}$-intertwining operators on $\cH^2_\lambda(D)$ under the strong operator topology.
\end{proposition}
\begin{proof}
Let $T \in \End_{\widetilde{H}}(\cH^2_\lambda(D))$ be given.

Because $H$ is a compact group, we see that $\cH^2_\lambda(D)$ decomposes into an orthogonal direct sum of finite-dimensional $\widetilde{H}$-invariant subspaces. This follows from the fact that we have $\widetilde{H} = H' Z(\widetilde{H})$ where $H'$ is a compact subgroup and $Z(\widetilde{H})$ acts by multiplication by elements in $\T$.

Let $W$ be a finite-dimensional $\widetilde{H}$-invariant subspace of $\cH^2_\lambda(D)$, and choose a basis $\{f_1,\ldots f_k\}$ for $W$.  By Theorem 2 in \cite{Englis}, it follows that there is a Toeplitz operator $T_{\varphi}$ with smooth, compactly-supported symbol $\varphi$ such that
\[
    \langle Tf_i, f_j \rangle = \langle T_\varphi f_i, f_j \rangle
\]
for $1\leq i,j \leq k$.  In particular,
\[
    \langle Tf, g \rangle = \langle T_\varphi f, g\rangle
\]
for all $f,g$ in $W$.

For any $h\in \widetilde{H}$ we have that
\begin{align*}
   \langle T^{(\lambda)}_{\varphi_h} \pi_\lambda(h)f, g\rangle
                 & = \langle \pi_\lambda(h) T^{(\lambda)}_\varphi f,g\rangle \\
                 & = \langle T^{(\lambda)}_\varphi f, \pi_\lambda(h^{-1})g \rangle \\
                 & = \langle T f, \pi_\lambda(h^{-1})g \rangle \quad\text{(since }W\text{ is }\widetilde{H}\text{-invariant)}\\
                 & = \langle \pi_\lambda(h)Tf, g\rangle \\
                 & = \langle T \pi_\lambda(h) f, g \rangle,
\end{align*}
Since $\pi_\lambda(h)$ is unitary and $f, g\in W$ were arbitrary, we see that
\[
   \langle T f, g \rangle = \langle T^{(\lambda)}_{\varphi_h} f, g\rangle
\]
for all $h \in \widetilde{H}$ and $f, g \in W$. By Lemma~\ref{lem:intertwiningToeplitz} we obtain for every $f,g \in \cH^2_\lambda(D)$
\[
   \langle T f, g \rangle
   = \langle \pi_\lambda(h) \circ T^{(\lambda)}_\varphi \circ \pi_\lambda(h)^{-1} f, g\rangle
\]
which by the remarks in this section can be considered as an identity that holds for every $h \in H$. Hence, we can integrate over $H$ and use Lemma~\ref{lem:averagingToeplitz} to conclude that
\[
   \langle T f, g \rangle = \langle T^{(\lambda)}_{\widehat{\varphi}} f, g\rangle
\]
for all $f, g\in W$. Note that $\widehat{\varphi}$ is an $H$-invariant symbol.

Because $W$ is an arbitrary finite-dimensional $\widetilde{H}$-invariant subspace and every cyclic representation of $\widetilde{H}$ is finite-dimensional, it follows that the space $\cT^{(\lambda)}(\cA^H)$ of Toeplitz operators with $H$-invariant symbols is dense in $\End_{\widetilde{H}}(\cH^2_\lambda(D))$ in the weak operator topology. Density in the strong operator topology follows because $\cT^{(\lambda)}(\cA^H)$ is a linear subspace and in particular convex.
\end{proof}

We now consider the following elementary result that we include for the sake of completeness.

\begin{lemma}
\label{lem:CommutingStrongClosure}
Suppose that $V$ is a linear subspace of the space $\cL(\cH)$ of bounded linear operators on a Hilbert space $\cH$.  If $V$ consists of operators which commute, then the closure $W$ of $V$ in the strong operator topology also consists of operators which commute.
\end{lemma}
\begin{proof}
First we show that $S \widetilde{T} = \widetilde{T} S$ for any $S\in W$ and $\widetilde{T} \in V$.

Fix $v\in \cH$; we will show that $S\widetilde{T}v = \widetilde{T}Sv$.
Because $V$ is dense in $W$ in the strong operator topology, there exists for any $\epsilon >0$ an operator $\widetilde{S}$ in $V$ such that
\[
   ||(\widetilde{S}-S)v|| < \epsilon \text{ and }
                ||(\widetilde{S}-S)(\widetilde{T}v)||<\epsilon.
\]
It then follows that
\begin{align*}
||(S\widetilde{T} - \widetilde{T}S)v|| &
     \leq ||(S\widetilde{T} - \widetilde{S}\widetilde{T})v|| +
          ||(\widetilde{S}\widetilde{T} - \widetilde{T}S)v||  \\
   & = ||(S - \widetilde{S})\widetilde{T}v|| +
          ||(\widetilde{T}(\widetilde{S} - S)v||  \\
   & \leq \epsilon + ||\widetilde{T}||\epsilon
\end{align*}
where we use the fact that operators in $V$ commute.  Because $\epsilon >0$ was arbitrary, it follows that $S\widetilde{T}v = \widetilde{T}Sv$.

Now let $S,T\in W$.  Fix $v\in V$.  Because $V$ is strongly dense in $W$, there is for each $\epsilon>0$ an operator $\widetilde{T}\in V$ such that
\[
   ||(\widetilde{T}-T)v|| < \epsilon \text{ and }
                ||(\widetilde{T}-T)(Sv)||<\epsilon.
\]
Then we have
\begin{align*}
||(ST-TS)v|| &
     \leq ||(ST - S\widetilde{T})v|| +
          ||(S\widetilde{T} - TS)v||  \\
   & = ||S(T - \widetilde{T})v|| +
          ||(\widetilde{T}-T)Sv||  \\
   & \leq ||S||\epsilon + \epsilon,
\end{align*}
where we use the fact that $S\widetilde{T} = \widetilde{T}S$.
Because $\epsilon>0$ was arbitrary, we see that $STv  = TS v$.
\end{proof}

The following is a characterization of the compact subgroups of $G$ for which the corresponding invariant symbols yield commutative $C^*$-algebras of Toeplitz operators.

\begin{theorem}
    \label{thm:commToeplitz-compact}
    Suppose that $H$ is a compact subgroup of $G$. Then, the following conditions are equivalent for every $\lambda > p - 1$.
    \begin{enumerate}
        \item The $C^*$-algebra $\cT^{(\lambda)}(\cA^H)$ is commutative.
        \item The restriction $\pi_\lambda|_{\widetilde{H}}$ is multiplicity-free.
    \end{enumerate}
\end{theorem}
\begin{proof}
    By Theorem~\ref{thm:commToeplitz-multiplicityFree} condition (2) implies condition (1).

    If (1) holds, then the space of Toeplitz operators with $H$-invariant symbols is commutative. By Proposition~\ref{prop:DenseInvariantToeplitz} and Lemma~\ref{lem:CommutingStrongClosure} it follows that the space $\End_{\widetilde{H}}(\cH^2_\lambda(D))$ is commutative as well. By Proposition~\ref{prop:typeImultiplicityFree} it follows that $\pi_\lambda|_{\widetilde{H}}$ is multiplicity-free.
\end{proof}

\begin{example}
\label{ex:torusSU(n,m)}
We now follow the notation from Example~\ref{ex:SU(n,m)}. Recall that we can take as a Cartan subalgebra the subalgebra of diagonal matrices in $\su(n,m)$
\[
    \gt =\left\{ d(\bs,\bt)
    \,\left|\, \bs \in \R^n,\; \bt \in \R^m,\; \sum_{j=1}^n s_j = -\sum_{k=1}^m t_k
    \right.\right\},
\]
where we denote
\[
    d(\bs,\bt)=i \sum_{j=1}^n s_jE_{jj} + i \sum_{j=1}^m t_jE_{n+j,n+j}
\]
and $E_{jk}$ denotes the matrix in $M_{(n+m)\times(n+m)}(\C)$ which is everywhere $0$ except at the position $(j,k)$ where it is $1$.

For $H$ the maximal torus with Lie algebra $\gt$, we therefore get the action
\[
    e^{-d(\bs,\bt)}Z = (e^{i(t_k -s_j)}z_{jk})_{jk}
\]
for every $Z = (z_{jk})_{jk} \in \gp_+ \simeq M_{n\times m}(\C)$. Hence, for every $\alpha \in \N^{n\times m}$ the action on the monomial function $Z^\alpha$ is given by
\[
    \pi_\lambda(e^{d(\bs,\bt)}) Z^\alpha = \exp\left(i\sum_{j=1}^n\sum_{k=1}^m(t_k - s_j)\alpha_{jk}\right)Z^\alpha.
\]
In particular, the $H$-action on the monomial $Z^\alpha$ corresponds to the character
\begin{align*}
    \chi_\alpha(d(\bs,\bt)) = \exp\left(-\sum_{j=1}^n\left(\sum_{k=1}^m \alpha_{jk}\right) s_j
                + \sum_{k=1}^m\left(\sum_{j=1}^n \alpha_{jk}\right) t_k\right).
\end{align*}
This can be used to see that the $H$-action has higher multiplicity for some $n,m\geq 2$.

For example, for $n = m = 2$ and a given $\alpha \in \N^{2 \times 2}$ we have $\chi_\alpha = \chi_\beta$ for any $\beta$ given by
\[
    \beta = \alpha +
        \begin{pmatrix}
            q & -q \\
            -q & q
        \end{pmatrix}
\]
such that $0 \leq q \leq \min(\alpha_{12},\alpha_{21})$ is an integer. As a consequence of Theorem~\ref{thm:commToeplitz-maxCompact} we conclude that for $H$ the maximal torus of diagonal matrices in $\SU(2,2)$ the $C^*$-algebra $\cT^{(\lambda)}(\cA^H)$ is noncommutative on $\cH^2_{(\lambda)}(D_{2,2}^I)$ for every $\lambda > 3$.

Noncommutative $C^*$-algebras generated by Toeplitz operators can be similarly obtained for other values of $n,m \geq 2$.
\end{example}


\begin{thebibliography}{XX}
\bibitem{ADO} M. Aristidou, M. Davidson and G. \'Olafsson, \textit{Laguerre functions on symmetric cones and recursion relations in the real case}, Journal of Computational and Applied Mathematics \textbf{199} (2007), No. 1, 95–112.

\bibitem{BVQuasiRadial} W. Bauer and N. Vasilevski, \textit{On the structure of a commutative Banach algebra generated By Toeplitz operators with quasi-radial quasi-homogeneous symbols}, Integral Equations Operator Theory \textbf{74} (2012), No. 2, 199--231.

\bibitem{Berezin} F. A. Berezin, \textit{Quantization in complex symmetric spaces}, (Russian) Izv. Akad. Nauk. SSSR Ser. Mat.  \textbf{39} (1975), No. 2, 363–-402.

\bibitem{Berger} M. Berger, \textit{Les espaces sym\'etriques noncompacts}, Annales Scientifique de l'\'Ecole Normale Sup\'erieure \textbf{74} (1957), No. 3, 85--177.

\bibitem{DGZ} M. Davidson and G. \'Olafsson and G. Zhang, \textit{Laplace and Segal-Bargmann transforms on Hermitian symmetric spaces and orthogonal polynomials}, Journal of Functional Analysis \textbf{204} (2003), No. 1, 157--195.

\bibitem{Englis} M. Engli\v{s},  \textit{Density of algebras generated by Toeplitz operators on Bergman spaces}, Arkiv f\"{o}r Matematik \textbf{30} (1992), No. 2, 227--243.

\bibitem{EnglisUpmeier} M. Engli\v{s} and H. Upmeier, \textit{Toeplitz quantization and asymptotic expansions for real bounded symmetric domains}, Mathematische Zeitschrift \textbf{268} (2011), 931--967.

\bibitem{FarautKoranyi} J. Faraut and A. Koranyi, \textit{Functions spaces and reproducing kernels on bounded symmetric domains}, Journal of Functional Analysis \textbf{88} (1990), No. 1, 64--89.

\bibitem{GKVRadial} S. Grudsky, A. Karapetyants and N. Vasilevski, \textit{Toeplitz operators on the unit ball in $\C^n$ with radial symbols}, Journal Operator Theory \textbf{49} (2003), No. 2, 325--346.

\bibitem{GKVHyperbolic} S. Grudsky, A. Karapetyants and N. Vasilevski, \textit{Dynamics of properties of Toeplitz operators on the upper half-plane: Hyperbolic case}, Boletin de la Sociedad Matematica Mexicana \textbf{10} (2004), 119--138.

\bibitem{GKVParabolic}  S. Grudsky, A. Karapetyants and N. Vasilevski, \textit{Dynamics of properties of Toeplitz operators on the upper half-plane: Parabolic case}, Journal of Operator Theory \textbf{52} (2004), No. 1, 185--204.

\bibitem{GQV} S. Grudsky, R. Quiroga-Barranco, and N. Vasilevski, \textit{Commutative $C^*$-algebras of Toeplitz operators and quantization on the unit disk}, Journal of Functional Analysis \textbf{234} (2006), No. 1, 1--44.

\bibitem{HarishChandraIV}  Harish-Chandra, \textit{Representations of semisimple Lie groups. IV}, American Journal of Mathematics \textbf{77} (1955), 743--777.

\bibitem{HarishChandraSpherical1} Harish-Chandra, \textit{Spherical functions on a semisimple Lie group. I}, American Journal of Mathematics \textbf{80} (1958), 241--310.

\bibitem{HarishChandraSpherical2} Harish-Chandra, \textit{Spherical functions on a semisimple Lie group. II}, American Journal of Mathematics \textbf{80} (1958), 553--613.

\bibitem{Helgason1970} S. Helgason, \textit{A duality for symmetric spaces with applications to group representations}, Advances in Mathematics \textbf{5} (1970), 1--154.

\bibitem{HelgasonBook} S. Helgason, \textit{Differential geometry, Lie groups and symmetric spaces}, Graduate Studies in Mathematics, 34. American Mathematical Society, Providence, RI, 2001.

\bibitem{HilgertOlafsson} J. Hilgert and G. \'Olafsson, \textit{Causal Symmetric Spaces, Geometry and Harmonic Analysis}, Perspectives in Mathematics, Vol. 18, Academic Press, New York, 1997.

\bibitem{Knapp} A.W. Knapp, \textit{Lie groups beyond an introduction}, Second edition. Progress in Mathematics, \textbf{140}. Birkh\"auser Boston, Inc., Boston, MA, 2002. 812 pp.

\bibitem{Kobayashi} T. Kobayashi, \textit{Multiplicity-free theorems of the restrictions
    of unitary highest-weight modules with respect to reductive symmetric pairs}, ``Representation Theory and Automorphic Forms'', Progress in Mathematics \textbf{255} (2008), 45--109.

\bibitem{Ol1991} G. \'Olafsson, \textit{Symmetric Spaces of Hermitian type}, Differential Geometry and its Applications \textbf{I} (1991), 195--233.

\bibitem{Ol2000} G. \'Olafsson, \textit{Analytic continuation in representation theory and harmonic analysis}. Global analysis and harmonic analysis (Marseille-Luminy, 1999), 201--233, S\'emin. Congr. \textbf{4}, Soc. Math. France, Paris, 2000.

\bibitem{OlafssonOrsted} G. \'Olafsson and B. {\O}rsted, \textit{Generalizations of the Bargmann transform}, Lie theory and its applications in physics (Clausthal, 1995), 3--14, World Sci.Publ., River Edge, NJ, 1996.

\bibitem{QSProjective}  Raul Quiroga-Barranco and A. Sanchez-Nungaray, \textit{Commutative $C^*$-algebras of Toeplitz operators on complex projective spaces}, Integral Equations Operator Theory \textbf{71} (2011), No. 2, 225--243.

\bibitem{QVReinhardt}  Raul Quiroga-Barranco and N. Vasilevski, \textit{Commutative algebras of Toeplitz operators on the Reinhardt domains}, Integral Equations Operator Theory \textbf{59} (2007), No. 1, 67--98.

\bibitem{QVBall1} R. Quiroga-Barranco and N. Vasilevski, \textit{Commutative $C^*$-algebras of Toeplitz operators on the unit ball. I. Bargmann-type transforms and spectral representations of Toeplitz operators}, Integral Equations Operator Theory \textbf{59} (2007), No. 3, 379--419.

\bibitem{QVBall2} R. Quiroga-Barranco and N. Vasilevski, \textit{Commutative $C^*$-algebras of Toeplitz operators on the unit ball. II. Geometry of the level sets of symbols}, Integral Equations Operator Theory \textbf{60} (2008), No. 1, 89--132.

\bibitem{Satake} I. Satake, \textit{Algebraic structures of symmetric domains}, Princeton University Press, Princeton, N.J., 1980.

\bibitem{VasilevskiTube} N. Vasilevski, \textit{The Bergman space in tube domains, and commuting Toeplitz operators}, Dokl. Akad. Nauk \textbf{372} (2000), No. 1, 9--12.

\bibitem{VasilevskiBook} N. L. Vasilevski, \textit{Commutative algebras of Toeplitz operators on the Bergman space}, Operator Theory: Advances and Applications, \textbf{185} Birkh\"auser Verlag, Basel, 2008.

\bibitem{Wallach1} N. R. Wallach, \textit{The analytic continuation of discrete series I}, Transactions of the American Mathematical Society \textbf{251} (1979), 1--17.

\end{thebibliography}
\end{document}